\documentclass[10pt]{amsart}

\usepackage{amssymb}
\usepackage{amsmath}
\allowdisplaybreaks
\usepackage{amsfonts}
\usepackage{fancybox}
\usepackage{empheq}
\usepackage{mathtools}

\usepackage[table]{xcolor}

\usepackage{color}
\usepackage{longtable,booktabs}
\usepackage{calrsfs} 
\usepackage{tikz-cd}
\usepackage[all,cmtip]{xy}

\usepackage{todonotes}
\usepackage{cleveref}

\definecolor{LightCyan}{rgb}{0.88,1,1}
\definecolor{LightCyan1}{rgb}{0.80,1,1}
\definecolor{Gray}{gray}{0.9}
\definecolor{Gray1}{gray}{0.95}

\usepackage{tikz}

\usepackage{changes}
\definechangesauthor[color=orange,name={Aristides Kontogeorgis}]{AK}

\newcommand{\Z}{\mathbb{Z}}

\newcommand{\C}{\mathbb{C}}

\newcommand{\Q}{\mathbb{Q}}

\newtheorem{theorem}{Theorem}
\newtheorem{lemma}[theorem]{Lemma}

\newtheorem{proposition}[theorem]{Proposition}
\theoremstyle{definition}

\newtheorem{remark}[theorem]{Remark}
\newtheorem{definition}[theorem]{Definition}
\newtheorem{convention}[theorem]{Convention}

 \usepackage[theorems,breakable]{tcolorbox}
\tcbuselibrary{theorems}

 \newtcbtheorem[number within=section]{impdef}{Definition}%
 {colback=yellow!5,colframe=LightCyan!,fonttitle=\bfseries,coltitle=black}{th}

\title[Group actions on curves]{Group Actions on cyclic covers of the projective line}

\date{\today}

\author[A. Kontogeorgis]{Aristides Kontogeorgis*}
\address{Department of Mathematics, National and Kapodistrian  University of Athens
Pane\-pist\-imioupolis, 15784 Athens, Greece}
\email{kontogar@math.uoa.gr}

\author[P. Paramantzoglou]{Panagiotis Paramantzoglou }
\address{Department of Mathematics, National and Kapodistrian University of Athens
Pane\-pist\-imioupolis, 15784 Athens, Greece}
\email{pan\_par@math.uoa.gr}


\date \today

\makeatletter
\newcommand{\aprod}{\mathop{\operator@font \hbox{\Large$\ast$}}}
\makeatother

\begin{document}

\begin{abstract}
We use  tools from combinatorial  group theory in order to   study actions of three types on groups acting on a curve, namely the automorphism group of a compact Riemann surface, the mapping class group acting on a surface (which now is allowed to have some points removed) and the absolute Galois group $\mathrm{Gal}(\bar{\Q}/\Q)$ in the case of cyclic covers of the projective line. 
\end{abstract}

\maketitle


%
%
%
\section{Introduction}
There is a variety of groups that can act on a Riemann surface/algebraic curve over $\C$; the automorphism group, the mapping class group (here we might allow punctures) and  if the curve is defined over $\bar{\Q}$, then  the absolute Galois 
group $\mathrm{Gal}(\bar{\Q}/\Q)$ is also acting on the curve. Understanding the above  groups is a difficult problem and these actions provide information on both the  curve and the group itself. For all the groups mentioned above the action can often be understood in terms of linear representations, by allowing the group to act on vector spaces and modules related to the curve itself, as the (co)homology groups and section of holomorphic differentials.

For a compact Riemann surface $X$ the automorphism group $\mathrm{Aut}(X)$ consists of all  invertible maps $X\rightarrow X$ in the category of Riemann surfaces.

A compact Riemann surface minus a finite number of punctures can be also seen as a connected, orientable topological surface and the mapping class group $\mathrm{Mod}(X)$ can be considered acting on $X$. The  mapping class group is  the quotient 
\[
\mathrm{Mod}(X)=\mathrm{Homeo}^+(X)/\mathrm{Homeo}^0(X),
\]
where 
$\mathrm{Homeo}^+(X)$ is  the group of orientation preserving homeomorphisms of $X$ and  $\mathrm{Homeo}^0(X)$
is the connected component of the identity in the compact-open topology.

These actions of the above mentioned three types of groups seem totally unrelated and come from different branches of Mathematics. Recent progress in the branch of ``Arithmetic topology'' provide us with a complete different picture. 
First the group $\mathrm{Aut}(X)$ can be seen as a subgroup of $\mathrm{Mod}(X)$ consisting of ``rigid'' automorphisms. 

Y. Ihara in  \cite{Ihara1985-it}, \cite{IharaCruz}, proposed a method to treat elements in $\mathrm{Gal}(\bar{\Q}/\Q)$ as elements  in the automorphism group of the profinite free group. This construction is similar to the realization of braids as automorphisms of the free group. This viewpoint of elements in $\mathrm{Gal}(\bar{\Q}/\Q)$ as  ``profinite braids'' allows us to give a series of Galois representations similar to classical 
braid representations.

In this article we will focus on curves which are  cyclic ramified covers of the projective line. 
These curves form some of the  few examples of Riemann surfaces where explicit computations can be made. 

A ramified cover of the projective curve reduces to a topological cover, when the branch points are removed. By covering map theory these covers correspond to certain subgroups of the fundamental group of the projective line with branch points removed, which is a free group.  

The computation of homology groups can be done by abelianization of the fundamental group, which in turn can be computed using the Schreier lemma. This method of computation  provides us with a unified way to treat all the actions on curves, by seeing an element in these aforementioned groups as an automorphism of the corresponding fundamental group.

The authors find very interesting that this approach
provides us with a totally new method in order to study actions in the dual case, that is actions on global sections of holomorphic differentials $H^0(X,\Omega_X)$. When $G$ is the automorphism group, the determination of the $G$-module structure $H^0(X,\Omega_X)$ is a classical problem first posed by Hecke \cite{MR3069500}, which was solved by Chevalley and Weil \cite{Chevalley1934-eb} using character theory, when the characteristic of the field is zero.  

For the $\mathrm{Mod}(X)$ case, in \cite{McMullenBraidHodge} C. McMullen considered unitary representations of the braid group acting on global sections of differentials of  cyclic covers of the projective line. His result can be recovered by our homological computations by dualizing. This approach was also  mentioned in this article \cite[p. 914 after th. 5.5.]{McMullenBraidHodge}.
  We believe that the details of this computation are worth studying and are by no means  trivial.

Finally  the homology approach allows us to study the pro-$\ell$ analogue according to Ihara's point of view, and several classical notions like the homology intersection pairing can be generalized to the Weil pairing for the Tate module. This fits well with the ``arithmetic topology'' viewpoint, where notions from knot theory have an arithmetic counterpart, \cite{Morishita2011-yw}, \cite{MorishitaATIT}.


Let us now describe the results and the structure of the article. 
Section \ref{AIreps} is devoted to the  construction of Artin's and Ihara's representations. 
In section \ref{sec:FundGroupCcover} we compute the generators of the fundamental group of the open curves involved in this article. All information is collected in table \ref{Tab:hom} of page \pageref{Tab:hom}. 

%

We will make computations in several group algebras for  multiplicative groups. In order to avoid confusion we will denote by $\mathbf{Z}=\{t^a: a\in \mathbb{Z}\}$ and by $\mathbf{Z}_\ell=\{t^a: a\in \mathbb{Z}_\ell\}$, where $t$ is a formal parameter. These groups are isomorphic to  the groups $\Z$ and $\Z_\ell$. The group $\Z/n\Z=\langle \sigma \rangle$ is considered to be generated by the order $n$ element $\sigma$. 

Select a set $\Sigma$ consisted of $s$ points of $\mathbb{P}^1$. 
Let $C_s$ be a topological cover of $X_s=\mathbb{P}^1\backslash \Sigma$ with Galois group $\mathrm{Gal}(C_s/X_s)=\mathbf{Z}$, see definition \ref{defCs}. Let also $Y_n$ be a topological cover of $X_s$, covered by $C_s$, so that $\mathrm{Gal}(Y_n/X_s)=\Z/n\Z$.
We will denote by $\bar{Y}_n$ the complete algebraic curve corresponding to $Y_n$.

In section \ref{sec:Uniform-ram} we  investigate the decomposition of the homology groups as Galois modules and prove the following
\begin{theorem}
The homology groups for the cyclic covers $C_s$ (resp. $Y_n$) can be seen as Galois modules for the group $\mathbf{Z}$ (resp. $\Z/n\Z$) as follows:
\begin{align}
H_1(C_s,\Z) =R_0/R_0' &=\Z[\mathbf{Z}]^{s-2} =\Z[t,t^{-1}]^{s-2} 
\label{homology-decomposition}
\\
H_1(Y_n,\Z) =R_n/R_n' &=\Z[\Z/n\Z]^{s-2} \bigoplus \Z.
\nonumber
\end{align}
\end{theorem}

Cyclic covers with infinite  Galois group lead to the Burau representation which is discussed in \ref{sec:BurauDiscrete}. 
Similar to the discrete case, we have that $H_1(C_s,\Z_\ell)= \Z_\ell[\mathbf{Z}]^{s-2}$ but in order to have an action of the absolute Galois group, a larger space is required, namely the completed group algebra
$\Z_\ell[[\mathbf{Z}_\ell]]^{s-2}$.

In section \ref{sec:Burau-prof} we give a pro-$\ell$ version of the analogue of a Burau representation
\[
\rho_{\mathrm{Burau}}: \mathrm{Gal}(\bar{\Q}/\Q) \rightarrow \mathrm{GL}_{s-2}(\Z_\ell[[\mathbf{Z}_\ell]])
\]
and in theorem \ref{MatBurau} we give a matrix expression of this representation. 


In section \ref{sec:applications-cyc-cov} for the complete curve $\bar{Y}_n$ we prove the following 
\begin{theorem}
Let $\sigma$ be a generator of the cyclic group $\Z/n\Z$.
The complete curve $\bar{Y}_n$ has homology 
\[
H_1(\bar{Y}_n,\Z)=J_{\Z/n\Z}^{s-2}, 
\]
 where 
$J_{\Z/n\Z}=\Z[\Z/n\Z]/\langle \sum_{i=0}^{n-1} \sigma^i\rangle$ is the co-augmentation module of  $\Z[\Z/n\Z]$.
\end{theorem}
The later space when tensored with $\C$ gives a decomposition 
\[
H_1(\bar{Y}_n,\Z)\otimes_\Z \C= \bigoplus_{\nu=1}^{n-1} V_\nu,
\]
where each $V_\nu$ is the $s-2$-dimensional eigenspace corresponding to eigenvalue $e^{\frac{2\pi i \nu}{n}}$ of the action of a  generator $\sigma$ of the group $\Z/n\Z$, where $\sigma$ is seen as a linear operator acting on $H_1(\bar{Y}_n,\Z)\otimes_\Z \C$.
Each space $V_\nu$ gives rise to a representation of the braid group $B_s$, which is the reduction of the Burau representation at $t\mapsto e^{\frac{2\pi i \nu}{n}}$. 

If $n=\ell^k$ then a similar reduction process can be applied to the  pro-$\ell$ Burau representation. We consider the $\ell^k-1$ non-trivial $\ell^k$-roots  of unity, $\zeta_1,\ldots,\zeta_{\ell^k-1}$  in the algebraically closed field $\bar{\Q}_\ell$.
We have
\[
\Z_{\ell}[[\mathbf{Z}_\ell]]^{s-2}
\otimes_{\Z_\ell} \bar{\Q}_\ell = \bigoplus_{\nu=1}^{\ell^k-1} V_\nu,
\]
which after reducing $\Z_{\ell}[[\mathbf{Z}_\ell]] \rightarrow \Z_{\ell} [\Z_{\ell}/\ell^k \Z_\ell]=\Z_{\ell} [\Z/\ell^k\Z]$ sending $t\mapsto \zeta_\nu$ gives rise to the representation in $V_\nu$. The modules $V_\nu$ in the above decomposition are only $\Z_\ell[[\mathbf{Z}_\ell]]$-modules and $\mathrm{ker}N$-modules, where $N:\mathrm{Gal}(\bar{\Q}/\Q)\rightarrow \Z_\ell^*$ is the pro-$\ell$ cyclotomic character. 






We would like to point out that the space $\Z_\ell[[\mathbf{Z}_\ell]]^{s-2}$ contains information of all covers $\bar{Y}_{\ell^k}$ for all $k\in \mathbb{N}$, and equals   the \'etale homology of a curve $\tilde{Y}$, which appears as a $\Z_\ell$-cover of the projective line, minus the same set of points removed. Going back from the arithmetic to topology we can  say that the classical discrete Burau representation can be recovered by all representations of finite cyclic covers $\bar{Y}_n$, since we can define the inverse limit of all mod $n$ representations obtaining the $B_s$-module $\Z[[\hat{\Z}]]^{s-2}$. This $B_s$-module in turn contains $\Z[\mathbf{Z}]^{s-2}$ as a dense subset.

Finally in section \ref{arithInter} we see how the analogue of the homology intersection pairing can be interpreted as an intersection pairing using the Galois action on the Weil pairing for the Tate module. 
For a free $\Z$ (resp. $\Z_\ell$)-module of rank $2g$, endowed with a symplectic pairing $\langle\cdot,\cdot\rangle$ the symplectic group is defined as
\[
\mathrm{Sp}(2g,\Z)=\{M\in \mathrm{GL}(2g,\Z):
\langle M v_1,M v_2\rangle=\langle v_1,v_2 \rangle
\}
\]
and the generalized symplectic group is defined as
\[
\mathrm{GSp}(2g,\Z)=\{M\in \mathrm{GL}(2g,\Z_\ell):
\langle M v_1,M v_2\rangle=m\langle v_1,v_2 \rangle, \text{ for some } m\in \Z_\ell^*
\}.
\]
In the topological setting the pairing is the intersection pairing and we have the following representation
\[
\rho: B_{s-1} \rightarrow \mathrm{Sp}(2g,\mathbb{Z})
\]
We employ properties of the Weil pairing in order to show that we have a representation  
\[
\rho': \mathrm{Gal}(\bar{\Q}/\Q) \rightarrow 
\mathrm{GSp}(2g,\Z_\ell)
\]
as an arithmetic analogue of the braid representation $\rho$.


{\bf Acknowledgement: }
The authors would like to thank Professor Nondas Kecha\-gias and the anonymous referee for their valuable comments and corrections.

\section{On Artin and Ihara representations}
\label{AIreps}
\subsection{Artin representation}

 It is known that the braid group  can be seen as an automorphism group of the free group  $F_{s-1}$ in terms of the Artin representation. More precisely the group $B_{s-1}$ can be defined as the subgroup of $\mathrm{Aut}(F_{s-1})$ generated by the elements $\sigma_i$ for $1\leq i \leq s-2$, given by
\[
\sigma_i(x_k)=
\begin{cases}
x_k & \mbox{ if } k\neq i,i+1,\\
x_i x_{i+1} x_i^{-1} & \mbox{ if } k =i, \\
x_i & \mbox{ if } k=i+1.
\end{cases}
\]
The open disk with $s-1$ points removed is homeomorphic with the the projective line with infinity and $s-1$ points removed. In particular, these spaces have isomorphic fundamental groups. 
Indeed,  
the free group $F_{s-1}$ is the fundamental group of $X_s$ defined as 
\begin{equation} \label{Xs-def}
X_s=\mathbb{P}^1-\{P_1,\ldots,P_{s-1},\infty\}.
\end{equation}
 In this setting the group $F_{s-1}$ is given as:
\begin{equation} \label{free-quodis}
F_{s-1}=\langle x_1,\ldots,x_{s}  | x_1x_2\cdots x_{s}=1\rangle,
\end{equation}
the elements $x_i$ correspond to homotopy classes of loop circling once clockwise around each removed point $P_i$.

\begin{remark} \label{larger-action}
Notice that not only $B_{s-1}$ acts on $F_{s-1}$ but also $B_{s}$ acts on $F_{s-1}$. Indeed, for the extra generator $\sigma_{s-1} \in B_s-B_{s-1}$ 
we define
\begin{equation}
\label{ss1}
\sigma_{s-1}(x_i)=x_i \qquad \text{ for } 1\leq i \leq s-2
\end{equation}
\begin{equation}
\label{ss2}
\sigma_{s-1}(x_{s-1})= x_{s-1} x_s x_{s-1}^{-1}
=
x_{s-2}^{-1}x_{s-3}^{-1} \cdots x_{1}^{-1} x_{s-1}^{-1}
\end{equation}
and using eq. (\ref{ss1}), (\ref{ss2}) we compute
\[
\sigma_{s-1}(x_s)=\sigma_{s-1}
(x_{s-1}^{-1}\cdots x_1^{-1})= 
 \sigma_{s-1}(x_{s-1})^{-1} 
 \big(
  x_{s-2}^{-1}\cdots x_1^{-1}
  \big)=
x_{s-1}.
\]
\end{remark}
\subsection{Ihara representation}
We will  follow the notation of  \cite{MorishitaATIT}. 
Y. Ihara, by considering the \'etale (pro-$\ell$) fundamental group of the space $\mathbb{P}^1_{\bar{\Q}}-\{P_1,\ldots,P_{s-1},\infty\}$, with $P_i\in \Q$,   introduced the monodromy representation
\[
\mathrm{Ih}_S: \mathrm{Gal}(\bar{\Q}/\Q) \rightarrow \mathrm{Aut}(\mathfrak{F}_{s-1}), 
\]
where $\mathfrak{F}_{s-1}$ is the pro-$\ell$ completion of the free group $F_{s-1}$. 
Here the  group $\mathfrak{F}_{s-1}$  admits a presentation, similar to eq. (\ref{free-quodis}),
\begin{equation}
\label{Frpres}
\mathfrak{F}_{s-1}=
\left\langle
x_1,\ldots,x_s | x_1x_2\cdots x_s=1
\right\rangle,
\end{equation}
where here $\mathfrak{F}_{s-1}$ is considered as a quotient of the free pro-$\ell$ group $\mathfrak{F}_s$ in the pro-$\ell$ category.

The image of the Ihara representation is inside the group \[
\tilde{P}(\mathfrak{F}_{s-1}):=
\left\{
\sigma\in \mathrm{Aut}(\mathfrak{F}_{s-1})| \sigma(x_i)\sim x_i^{N(\sigma)}
(1\leq i \leq s) \text{ for some } N(\sigma) \in \Z_\ell^*
\right\},
\]
where $\sim$ denotes the conjugation equivalence. 
This group is the arithmetic analogue of the Artin representation of ordinary (pure) braid groups inside $\mathrm{Aut}(F_{s-1})$.
Notice that the exponent 
$N(\sigma)$
depends only on  $\sigma$ and not on $x_i$. Moreover the map
\[
	N: \tilde{P}(\mathfrak{F}_{s-1})
	\rightarrow \mathbb{Z}_\ell^*
\]
is a group homomorphism and 
$N\circ\mathrm{Ih}_S:
\mathrm{Gal}(\bar{\Q}/\Q)\rightarrow \mathbb{Z}_\ell^*$ coincides with the cyclotomic character $\chi_\ell$.
\begin{remark}
As in remark \ref{larger-action}
the relation $x_1\cdots x_{s-1} x_s=1$ implies that $\tilde{P}(\mathfrak{F}_{s-1})$ also acts on the free group $\mathfrak{F}_s$ since $x_s=(x_1 \cdots x_{s-1})^{-1}$.
\end{remark}
In this setting an element $\sigma\in \mathrm{Gal}(\bar{\Q}/\Q)$ can be seen acting on the topological generators $x_1,\ldots,x_{s-1}$ of the free group by 
\begin{equation}
\label{actGeneratos}
\sigma(x_i)=w_i(\sigma) x_i^{N(\sigma)} w_i(\sigma)^{-1}.
\end{equation}
Moreover, by normalizing by an inner automorphism we might assume that $w_1(\sigma)=1$. We will use this normalization from now on.  

\begin{remark}
We have considered in Ihara's representation the points $P_1,\ldots,P_{s-1}$ to be in $\mathbb{Q}$. 
If we allow $P_1,\ldots,P_{s-1}$ to be in $\bar{\Q}$ then there is a minimal algebraic number field $K$which contains them all. We can consider in exactly the same way the absolute Galois group $\mathrm{Gal}(\bar{\Q}/K)=\mathrm{Gal}(\bar{K}/K)$ and then all arguments of this article work in exactly the same way for $\mathrm{Gal}(\bar{K}/K)$.

If now we want to consider representations of $\mathrm{Gal}(\bar{\Q}/\Q)$ but the field $K$ defined by the set of points $\bar{P}:=\{P_1,\ldots,P_{s-1}\}$ is strictly bigger than $\Q$, then  in order to obtain a reasonable action of $\mathrm{Gal}(\bar{\Q}/\Q)$ on the set of branch points, we have to assume that the polynomial 
$
f_{\bar{P}}(x):=\prod_{j=1}^{s-1} (x-P_j)$ is in $\Q[x].  
$
In this case the absolute Galois group induces a permutation action on the points $P_j$ and defines a subgroup of the the symmetric group $S_{s-1}$.


 The braid group is equipped by an onto map $\phi:B_{s-1} \rightarrow S_{s-1}$ with kernel the group of pure braids.  

 We have have argued that the braid group $B_s$ is a discrete analogue of the absolute Galois group $\mathrm{Gal}(\bar{\Q}/\Q)$. Every selection of points $\bar{P}$, which gives rise to a polynomial $f_{\bar{P}}(x)\in \Q[x]$, provides us with a map
  $\phi_{\bar{P}}:\mathrm{Gal}(\bar{\Q}/\Q)\rightarrow S_{s-1}$. We would like to see these maps $\phi_{\bar{P}}$ as analogues of the map $\phi$. The group of ``pure braids'' 
  with respect to such a map $\phi_{\bar{P}}$ is the 
  absolute Galois group $\mathrm{Gal}(\bar{\Q}/K)$ of the field $K$
  generated by the set of points $\bar{P}$, while the 
  image 
  $\phi_{\bar{P}}(\mathrm{Gal}(\bar{\Q}/\Q)) \subset S_{s-1}$ 
  is not onto, unless the points $P_1,\ldots,P_{s-1}$ 
  have no polynomial algebraic relations defined over $\Q$. As a matter of fact it conjectured -this is the inverse Galois problem- that any finite group can appear as the image of such a map $\phi_{\bar{P}}$ allowing $\bar{P}$ and $s$ to vary.  
In this way we obtain a short exact sequence 
\[
1 \rightarrow 
\mathrm{Gal}(\bar{\Q}/K) 
\rightarrow \mathrm{Gal}(\bar{\Q}/\Q) 
\stackrel{\phi_{\bar{P}}}{\longrightarrow} 
\mathrm{Gal}(K/\Q) \rightarrow 1.
\]
In general case, even if $\bar{P}$ is not a subset of $\Q$, there is a representation 
\[
\rho: \mathrm{Gal}(\bar{\Q}/\Q) \rightarrow 
\mathrm{Aut}(\mathfrak{F}_{s-1})
\] 
where for $\sigma \in \mathrm{Gal}(\bar{\Q}/\Q)$ $\rho(\sigma)(x_i)=w(\sigma) x_{\phi_{\bar{P}}(\sigma)} w(\sigma)^{-1}$. If $\sigma$ is a ``pure braid'',  then
the above action can be simplified, since the generator $x_i$ of $\mathfrak{F}_{s-1}$ is not moved to another generator.  
For this article the interesting part is the study of $\mathrm{Gal}(\bar{\Q}/\Q)$ and not the problem of finding the Galois group of a polynomial in $\Q[x]$. 
If we start by selecting all points in $\bar{P}$ in $\Q$, as Ihara did,  then the whole group $\mathrm{Gal}(\bar{\Q}/\Q)$ can be considered as an analogue of pure braids. 

\end{remark}

\subsection{Similarities}

For understanding representations of the absolute Galois group $\mathrm{Gal}(\bar{\Q}/\Q)$, the theory of coverings of $\mathbb{P}^1_{\Q}-\{0,1,\infty\}$ is enough, by Belyi's theorem, \cite{Belyi1}.
On the other hand the study of topological covers
of $\mathbb{P}^1_{\Q}-\{0,1,\infty\}$ is not very interesting; both groups 
$B_2$ and $B_3$ which can act on covers of $\mathbb{P}^1_{\Q}-\{0,1,\infty\}$ are  not very interesting braid groups. 
In order to seek out similarities between the Artin and Ihara representation, we will study   covers with more than three points removed.  
Notice that when the number $s$ of points we remove is $s>3$, then we expect that their configuration might also affect our study.

Moreover elements in the braid group are acting like elements in the mapping class group of the punctured disk i.e. on the  projective line minus $s$ points. 
The braid group acts like the symmetric group on the set of removed points $\Sigma$ and acts like a complicated homeomorphism on the complement $D_{s-1}$ of the $s-1$ points.

Let $\Sigma=\bar{P} \cup \{\infty\}$ and let $K$ be the field generated by the points in $\bar{P}$. The group $\mathrm{Gal}(\bar{\Q}/K)$ keeps invariant the 
set $\Sigma$ and corresponds to the notion of pure braids. Since $\mathrm{Gal}(\bar{\Q}/K)$ also acts on $\mathbb{P}^1_{\bar{\Q}}$ it acts on the difference 
$\mathbb{P}^1_{\bar{\Q}}\backslash \Sigma$. This mysterious action should be seen as the arithmetic analogue of the action of the (pure)braid group on the punctured disc.

Knot theorists study braid group representations, in order to provide invariants of knots (after Markov equivalence, see \cite[III.6 p.54]{Prasolov97}) and number theorists study Galois representations in order to understand the absolute Galois group $\mathrm{Gal}(\bar{\mathbb{Q}}/\mathbb{Q})$. Both kind of representations are important and bring knot and number theory together within the theory of  arithmetic topology. 

\section{On the fundamental group of cyclic covers}
\label{sec:FundGroupCcover}

Let $\pi:Y\rightarrow \mathbb{P}^1$ be a ramified 
Galois cover of the projective line ramified above the set $\Sigma=\{P_1,\ldots,P_s\} \subset \mathbb{P}^1$. 
The open curve $Y_0=Y-\pi^{-1}(\Sigma)$ is then a topological cover of $X_s=\mathbb{P}^1-\Sigma$ and can be seen as a quotient of the universal covering space $\tilde{X}_s$ by the free subgroup $R_0=\pi_1(Y_0,y_0)$ of the free group $\pi_1(X_s,x_0)=F_{s-1}$ (resp. pro-$\ell$ free group $\mathfrak{F}_{s-1}$), where $s=\#\Sigma$. 
We will employ the Reidemeister Schreier method, 
algorithm \cite[chap. 2 sec. 8]{bogoGrp},\cite[sec. 2.3 th. 2.7]{MagKarSol} in order to compute the group $R_0$.

\subsection{Schreier' s Lemma}
Let $F_{s-1}=\langle x_1, \cdots, x_{s-1}  \rangle$ be the free group with basis $X=\{ x_1, \cdots, x_{s-1}\}$ and let $H$ be a
subgroup of of $F_{s-1}$.

 A (right) {\bf Schreier Transversal} for $H$ in $F_{s-1}$ is a set $T=\{t_1=1, \cdots, t_n \}$ of reduced words, such that each right coset of $H$ in $F_{s-1}$ contains a unique word of $T$ (called a representative of this class) and all 
 initial segments of these words also lie in $T$.
 In particular, $1$ lies in $T$ (and represents the class $H$) and $H t_i \neq H t_j$, $\forall i \neq j$. For any $g \in F_{s-1}$ denote by $\overline{g}$ the element of $T$ with the property $Hg=H\overline{g}$.

 If $t_i \in T$ has the decomposition as a reduced word
$t_i=x_{i_1}^{e_1} \cdots x_{i_k}^{e_k}$ (with $i_j=1, \ldots, s-1$, $e_j= \pm 1$ and $e_j=e_{j+1}$ if $x_{i_j}=x_{i_{j+1}})$, then for every word $t_i$ in $T$ we have that
\begin{equation} \label{Xs-def1}
t_i=x_{i_1}^{e_1} \cdots x_{i_k}^{e_k} \in T \Rightarrow 1, x_{i_1}^{e_1}, x_{i_1}^{e_1}x_{i_2}^{e_2},\ldots, x_{i_1}^{e_1}x_{i_2}^{e_2} \cdots x_{i_k}^{e_k} \in T.
\end{equation}

\begin{lemma}[Schreier's lemma]
\label{lemma:schreier}
Let $T$ be a right Schreier Transversal for $H$ in $F_{s-1}$ and set $\gamma(t,x):= tx \overline{tx}^{-1}$, $t \in T$, $x \in X$ and $tx \notin T$. Then $H$ is freely generated by the set 
\begin{equation} \label{free-quods}
\{ \gamma(t,x) | \gamma(t,x) \neq 1 \rangle
\}.
\end{equation}
\end{lemma}

\subsection{Automorphisms of Free groups acting on subgroups}
If $R_0=\pi_1(Y_0,y_0)$ is a characteristic subgroup of $F_{s-1}=\pi_1(\mathbb{P}^1-\Sigma)$ (resp. of $\mathfrak{F}_{s-1}$ in the pro-$\ell$ case) then it is immediate that the Artin (resp. Ihara) representation gives rise  to an action on $R_0$.

Observe that since the cover $\pi:Y \rightarrow \mathbb{P}^1$ is Galois we have that $R_0 \lhd F_{s-1}$ and the Artin representation gives rise to a well defined action of the braid group on $R_0$.


The same argument applies for the kernel of the norm map in the Ihara case, that is since the pro-$\ell$ completion of $R_0$ is a normal subgroup of $\mathfrak{F}_{s-1}$, 
every element $\sigma$ in $\mathrm{Gal}(\bar{\Q}/\Q)$ with $\chi_\ell(\sigma)=1$ acts on the pro-$\ell$ completion of $\pi_1(Y_0,y_0)$.  

This is  in accordance with a result of J. Birman and H. Hilden \cite[th. 5]{Birman1972-pg}, which in the case of cyclic coverings $\pi:C \rightarrow (\mathbb{P}^1-\Sigma)$, relates  
the subgroup $\mathrm{Mod}_{\pi}(C)$ of the mapping class group of $C$ consisted of the  fiber preserving automorphisms,  the Galois group $\mathrm{Gal}(C/\mathbb{P}^1)$ and the mapping class group $\mathrm{Mod}(\mathbb{P}^1-\Sigma)$ of $\mathbb{P}^1-\Sigma$ in terms of the quotient
\[
\mathrm{Mod}_\pi(C)/\mathrm{Gal}(C/\mathbb{P}^1)=\mathrm{Mod}(\mathbb{P}^1-\Sigma).
\]

For example when  $Y$ is the covering corresponding to the commutator group $F_{s-1}'$, then $\mathrm{Gal}(Y/X_s)\cong F_{s-1}/F_{s-1}'=H_1(X,\mathbb{Z})$. Therefore, the latter space is acted on by the group of automorphisms, and the braid group $B_s$. 

\subsection{Automorphisms of curves}
For the case of automorphisms of curves, where  the Galois cover $\pi:Y \rightarrow \mathbb{P}^1$ has  Galois group $H$, we consider
 the short exact sequence
\[
1 \rightarrow R_0 \rightarrow F_{s-1} \rightarrow H \rightarrow 1.
\]
We see that there is an action of $H$ on $R_0$ modulo inner automorphisms
of $R_0$ and in particular  a well defined action of $H$ on $R_0/R_0'=H_1(Y_0,\mathbb{Z})$.
Therefore the space $H_1(Y_0,\mathbb{Z})$ can be seen as a direct sum of indecomposable $\Z[H]$-modules.
\begin{remark}
A cyclic cover $X$ given  in eq. (\ref{cyccov}) might have a bigger automorphism group than the cyclic group of order $n$, if the roots $\{b_i, 1\leq i \leq s\}$ form a special configuration. Notice also that if the number $s$ of branched  points satisfies $s > 2n$ then the automorphism group $G$ fits in a short exact sequence 
\begin{equation} \label{extGrp}
1 \rightarrow \Z/n\Z \rightarrow G \rightarrow H \rightarrow 1,
\end{equation}
where $H$ is a subgroup of $\mathrm{PGL}(2,\C)$ \cite[prop. 1]{Ko:99}. The first author in \cite{Ko:99} classified all such extensions.

Observe that the action of the mapping class group on homology is of topological nature and hence independent of the special configuration of the roots $b_i$. If these roots have a special configuration, then certain elements of the mapping class group become automorphisms of the curve. This phenomenon is briefly explained on page 895 of \cite{McMullenBraidHodge}.

Similarly, suppose that the set $b_1,\ldots,b_s$ is fixed point wise by the absolute Galois group, that is $b_1,\ldots,b_s \in \mathbb{P}^1(\Q)$. The action of elements of  $\mathrm{Gal}(\bar{\Q}/\Q)$  on homology is the same for all such selections of $\{b_1,\ldots,b_s\} \subset \mathbb{P}^1(\Q)$. 
However if these roots $b_i$ have a special configuration, then certain elements of $\mathrm{Gal}(\bar{\Q}/\Q)$ become automorphisms of the group. 


If the branch locus $\{b_i:1\leq i \leq s\}$ is invariant under the group $H$ then $H_1(X,\Z)$ is a $\Z[G]$ module, where $G$ is an extension of $H$ with kernel $\Z/n\Z$ given by eq. (\ref{extGrp}). 
\end{remark}

\subsection{Adding the missing punctures}
Let us now relate the group $R=\pi_1(Y,y_0)$ corresponding to the complete curve $Y$ with the group 
$R_0$ corresponding to the open curve $Y_0=Y-\pi^{-1}(\Sigma)$. 
We  know that the group $R_0$ admits a presentation 
 \[
R_0=\langle a_1,b_1,\ldots, a_g,b_g,\gamma_1,\ldots,\gamma_s
| \gamma_1 \gamma_2 \cdots \gamma_s \cdot [a_1,b_1][a_2,b_2] \cdots
[a_g,b_g]=1
 \rangle,
 \]
 where $g$ is the genus of $Y$.

\begin{convention} 
Given $\gamma_1,\ldots,\gamma_s$ group elements   we will denote by $\langle \gamma_1,\ldots,\gamma_s \rangle$ the closed normal group generated by these elements. In the case of usual  groups the extra ``closed'' condition is automatically satisfied, since these groups have the discrete topology. So the ``closed group'' condition has a non-trivial  meaning only in the pro-$\ell$ case.
\end{convention}

 The completed curve $Y$ has a fundamental group which admits a presentation of the form
\begin{align*}
R & = \langle a_1,b_1,a_2,b_2,\ldots, a_g,b_g |  [a_1,b_1][a_2,b_2] \cdots
[a_g,b_g]=1
 \rangle
\\
 &=\frac{R_0}{\langle \gamma_1,\ldots,\gamma_s\rangle}.
\end{align*}
There is the following short exact sequence relating the two homology groups:
\begin{equation} \label{relate-gamma}
\xymatrix@R-15pt@C=10pt{
0 \ar[r] &
\langle \gamma_1, \ldots,\gamma_s \rangle \ar[r] &
H_1(Y_0,\mathbb{Z}) \ar[r] \ar[d]^{\cong} &
H_1(Y,\mathbb{Z}) \ar[r] \ar[d]^{\cong}
&
0\\
&  &   R_0/R_0' \ar[r] & R/R' =R_0/R_0' \langle \gamma_1,\ldots,\gamma_s \rangle & &
}
\end{equation}
Note that if a group acts on $R_0$,
then this action can be extended to an action of $R_0/\langle \gamma_1,\ldots,\gamma_s \rangle$ if and only if the group keeps $\langle \gamma_1,\ldots,\gamma_s \rangle$ invariant.

\section{Examples- Curves with punctures}
\label{sec:Uniform-ram}

\begin{definition} \label{defCs}

Recall that 
$X_s=\mathbb{P}^1\backslash \Sigma$, where $\Sigma$ is a subset of $\mathbb{P}^1$ consisted of $s$ points. 
Consider the projection
\[
0 \rightarrow I \rightarrow H_1(X_s,\mathbb{Z}) 
\stackrel{\alpha}{\longrightarrow} 
\mathbb{Z}
\rightarrow  
0
\]
and let $C_s$ be the  curve given as quotient $Y/I$, so that $\mathrm{Gal}(C_s/X_s)=\mathbb{Z}$. The map $\alpha$ is the winding number map which can be defined both on the fundamental group and on its abelianization by: $(1 \leq i_1,\ldots,i_t \leq s, \ell_{i_1},\ldots,\ell_{i_t} \in \Z)$
\begin{equation}
\label{w-map}
	\alpha:\pi_1(X_s,x_0)  \longrightarrow \mathbb{Z} \qquad
x_{i_1}^{\ell_{i_1}} x_{i_2}^{\ell_{i_2}} \cdots x_{i_t}^{\ell_{i_t}} \mapsto
\sum\limits_{\mu=1}^t \ell_{i_\mu}.
\end{equation}
\end{definition}

The following map is a  pro-$\ell$ version of the $w$-map defined in eq. (\ref{w-map}). Let $\mathfrak{F}_{s-1}$ be the free pro-$\ell$ group in generators  $  x_1,\ldots,x_{s-1}$.
Consider the map 
\begin{equation}
\label{a-map}
\alpha: \mathfrak{F}_{s-1}\rightarrow 
\mathfrak{F}_{s-1}/\langle x_1x_j^{-1}, j=2,\ldots,s-1 \rangle \cong
\mathfrak{F}_1\cong \Z_\ell.
\end{equation}
The map $\alpha$ is continuous so if $v_n$ is a sequence of words in $F_{s-1}$ converging to $v\in \mathfrak{F}_{s-1}$, then 
\[
\lim_n \alpha(v_n)= \alpha(v) \in \Z_\ell.
\]

\subsection{On certain examples of cyclic covers of $\mathbb{P}^1$}
Consider the  commutative diagram below on the left:  

\begin{minipage}{0.3\textwidth}
$
\xymatrix{
	\tilde{X}_s \ar[ddd]|-{F_{s-1}} \ar[drr]|-{F_{s-1}'} \ar[ddr]^{R_0} &  & \\
	 & &  Y \ar@/^2pc/[ddll]^{H_1(X_s,\mathbb{Z})} \ar[dl]_{I} \\
	  & C_s \ar[dl]_{\mathbf{Z}} & \\
	 X_s &
}
$
\end{minipage}
\begin{minipage}{0.65\textwidth}
Then $H_1(C_s,\mathbb{Z})=R_0/R_0'$, where $R_0=\pi_1(C_s)$ is the free subgroup of $F_{s-1}$ corresponding to $C_s$. Moreover $H_1(C_s,\mathbb{Z})$ is a free  $\mathbb{Z}[\mathbf{Z}]$-module free of rank $s-2$ acted on also by $B_{s-1}$ giving rise to the so called Burau representation:
\[
\rho: B_{s-1} \rightarrow \mathrm{GL}(s-2,\mathbb{Z}[t,t^{-1}]).
\]
 Keep in mind that $\mathbb{Z}[\mathbf{Z}]\cong \mathbb{Z}[t,t^{-1}]$. In what follows will give a proof of these facts using the Schreier's lemma.
\end{minipage}

\begin{lemma}
The group $R_0$, is an infinite rank group and is freely generated  by the set
\begin{equation}
\label{R-free-gen}
	\{x_1^i x_j x_1^{-i-1}: i\in \mathbb{Z}, j\in {2,\ldots, s-1} \}.
\end{equation}
\end{lemma}
\begin{proof}
Consider the epimorphisms
\[
	\xymatrix{
	F_{s-1} \ar[r]^-{p'} \ar@/_1.3pc/[rr]_{\alpha}& F_{s-1}/F_{s-1}' \ar[r]^-{p''} & \mathbb{Z}=H_1(Y,X_s)/I.
	}
\]
Set $\alpha=p''\circ p'$.
Let $y$ be an element in $\alpha^{-1}(1_\mathbb{Z})$. By the properties of the winding number we can take as $y=x_1$. Moreover $\alpha(x_j)=y$ for all $1\leq j \leq s-1$, since the
automorphism $x_i \leftrightarrow x_j$ is compatible with $I$ and therefore introduces an automorphism of $\mathbb{Z}$, so $\alpha(x_j)=y^{\pm 1}$, and we rename the generators $x_i$ to $x_i^{-1}$ if necessary.

Let $T:=\{y^i:i\in \mathbb{Z}\}\subset F_{s-1}$ be  a set of representatives of classes in $F_{s-1}/R_0\cong \mathbf{Z}$. The set $T$ is a  Schreier transversal, and Schreier's lemma can be applied, see lemma \ref{lemma:schreier}. For every $x\in F_{s-1}$ we will denote by $\bar{x}$ the representative in $T$. Moreover for all $i\in \mathbb{Z}$ and $1\leq j \leq s-1$ we have $\overline{y^i x_j}=y^{i+1}$ and by the Schreier's lemma  we see that
\[
y^i x_j \left(\overline{y^i x_j}\right)^{-1}=y^i x_j y^{-i-1}=
	x_1^i x_j x_1^{-i-1} 
	\qquad i\in \mathbb{Z}, j\in {2,\ldots, s-1} .
\]
\end{proof}
\begin{remark}  
The action of $\Z[\mathbf{Z}]$ on $R_0/R'_0$ is given by conjugation.
This means that for $n \in \Z$ we have
\begin{align}
\label{Zaction}
 \Z[\mathbf{Z}] \times R_0 & \longrightarrow R_0 \\
   (t^n, r) & \longmapsto x_1^n r x_1^{-n} \nonumber
\end{align}
A generating set for $H_1(C_s,\mathbb{Z})$ as a free $\Z[\mathbf{Z}]$-module  is given by the $s-2$ elements $\beta_j:=x_jx_1^{-1}$. Moreover the $\mathbf{Z}$-action is given by
\[
\left( x_ix_1^{-1}\right)^{t^n}=x_1^{n}x_ix_1^{-n-1},
\]
where  $t$ is a generator of the infinite cyclic group $\mathbf{Z}$. This means that 
 $H_1(C_s,\mathbb{Z})$ is a free $\mathbb{Z}[\mathbf{Z}]$-module of rank $s-2$.

Observe that in $R_0/R_0'$ we have 
\begin{align*}
x_j(x_i x_1^{-1}) x_j^{-1} &= 
(x_j x_1^{-1}) x_1\beta_i x_1^{-1} (x_j x_1^{-1})^{-1}\\
&= \beta_j x_1 \beta_i x_1^{-1} \beta_j^{-1} =\beta_i^{t},
\end{align*}
i.e. the conjugation by any generator $x_j$
has the same effect as the conjugation by $x_1$.
\end{remark}

\noindent
\begin{minipage}{0.57\textwidth}
Let us now consider a finite cyclic cover $Y_n$ of $X_s$ which is covered by $C_s$, i.e. we have the diagram on the right bellow:
\begin{lemma}
\label{lemma:funRn}
The group $R_n=\pi_1(Y_n)\supset R_0$ is the kernel of the map $\alpha_n$
\[
	\xymatrix{
	\pi_1(X) \ar[r]^-\alpha \ar@/^1pc/[rr]^{\alpha_n}& \mathbf{Z} \ar[r] & \mathbb{Z}/n\mathbb{Z}.
	}
\]
\end{lemma}
\begin{proof}
This is clear from the explicit description of the group $R_0$ given in eq. (\ref{R-free-gen}). 
\end{proof}
\end{minipage}
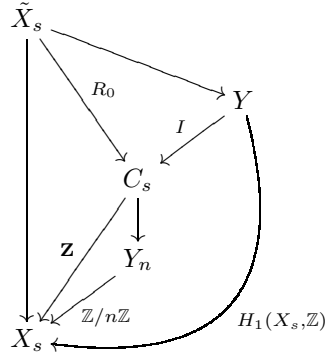
\begin{minipage}{0.4\textwidth}
$
\xymatrix@R=16pt{
	\tilde{X}_s \ar[dddd] \ar[drr] \ar[ddr]^{R_0} &  & \\
	 & &  Y \ar@/^4pc/[dddll]^{H_1(X_s,\mathbb{Z})} \ar[dl]_{I} \\
	  & C_s \ar[ddl]_{\mathbf{Z}}  \ar[d] & \\
      & Y_n \ar[dl]^{\mathbb{Z}/n\mathbb{Z}} \\
	 X_s &
}
$
\end{minipage}

\begin{lemma} \label{Rngenerators}
The group $R_n$ is generated by
\[
R_n=
\{
x_1^i x_j x_1^{-i-1}: 0 \leq i \leq n-2, 2\leq j\leq  s-1 \}
\cup
\{x_1^{n-1}x_j: 1\leq j \leq s-1\}.
\]
which is a free group on $r=(s-2)n+1$ generators.
\end{lemma}
\begin{proof}
In this case the transversal set equals $T=\{y^i: 0 \leq i \leq n-1\}$.
Moreover
\[
\overline{y^i x_j}=
\begin{cases}
y^{i+1} & \text{ if } i< n-1 \\
1 & \text { if } i=n-1.
\end{cases}
\]
For all $i$, $0 \leq i \leq n-1$ and for all generators $x_j$, $1\leq j \leq s-1$ we compute
\[
y^i x_j (\overline{y^i x_j})^{-1}=
\begin{cases}
y^i x_j y^{-i-1}=x_1^ix_jx_1^{-i-1} & \text{ if } 0\leq i \leq n-2 \\
y^{n-1} x_j=x_1^{n-1} x_j & \text{ if } i=n-1
\end{cases}
\]
Keep in mind that if $j=1$ then $x_1^{i}x_j x_1^{-i-1}=1$ and this value does not give us a generator. On the other hand the expression $x_1^{-1}x_j$ survives even if $j=1$. 
The desired result follows.
\end{proof}
\begin{proposition}
\label{prop:GalModRn}
The  $\mathbb{Z}$-module $R_n/R_n'$ as $\mathbb{Z}[\mathbb{Z}/n\mathbb{Z}]$-module is isomorphic to
\[
R_n/R_n' = \mathbb{Z}[\mathbb{Z}/n\mathbb{Z}]^{s-2} \bigoplus \mathbb{Z}.
\]
\end{proposition}
\begin{proof}
Set $\beta_j=x_j x_1^{-1}$ for $2\leq j \leq s-1$. Then the action of $\mathbb{Z}/n\mathbb{Z}=\langle \sigma \rangle$ on elements $\beta_j$ is 
given by
\[
\beta_j^{\sigma^\ell}=x_1^{\ell} 
\left(x_j x_1^{-1}
\right)
 x^{-\ell}=x_1^{\ell} x_j x_1^{-\ell-1}
\text{ for } 0 \leq \ell \leq n-1.
\]
It is clear that for each fixed $j$, $2\leq j \leq s-1$, the elements $\beta_j^{\sigma^\ell}$ generate a copy of the group algebra $\mathbb{Z}[\mathbb{Z}/n\mathbb{Z}]$. By the explicit form of the basis generators given in lemma \ref{Rngenerators} we have the alternative basis given by
\label{sec:Rn}
\begin{equation}
\label{secBaseRn}
\{ x_1^{i}x_jx_1^{-i-1}: 2 \leq j \leq s-1, 0\leq i \leq n-1\} \cup \{x_1^n\}.
\end{equation}
The result follows.
\end{proof}
\begin{remark}
The above computation is compatible with the Schreier index formula \cite[cor. 8.5 p.66]{bogoGrp} which asserts that
\begin{equation}
\label{ASRn}
r-1=n(s-2).
\end{equation}
\end{remark} 
\begin{remark}
Observe that there is no natural reduction modulo $n$ map from $H_1(C_s,\Z)$ to $H_1(Y_n,\Z)$ corresponding to the group reduction $\mathbf{Z} \rightarrow \Z/n\Z$.
\end{remark}

We collect here the generators of the open curves involved in this article. 
The curves on the third column correspond to the quotients of the universal covering space of $X_s$ by the groups of the first column. 
\begin{table}[h]
\begin{tabular}{|lllll|}
 \rowcolor{LightCyan}
 \hline
\textup{Group} & \textup{Generators} & \textup{Curve} & \textup{Galois group} 
& \textup{Homology} \\
 \hline
 \rowcolor{Gray1}
 $ F_{s-1}$ 
 & 
 $x_1,\ldots,x_{s-1}$
 & 
 $X_s$ 
 &
 $\{1\}$ & $F_{s-1}/F_{s-1}'$ 
 \\
 \rowcolor{Gray}
 $\{1\}$ & $\emptyset$ & $\tilde{X}_s$ & $F_{s-1}$ &  
 $\{1\}$
 \\
 \rowcolor{Gray1}
  $F_{s-1}'$ & $[x_i,x_j], i\neq j$ & $Y$  & $ F_{s-1}/F_{s-1}'$ & $F_{s-1}'/F_{s-1}''$
\\
\rowcolor{Gray}
$R_0$ & $ x_1^i x_j x_1^{-i-1}, 
\substack{
i\in \Z
\\
2\leq j \leq s-1}
$ & $C_s$ & $\mathbf{Z}$ & $R_0/R_0'$ 
\\
\rowcolor{Gray1}
$R_n$ &
$\begin{array}{l} 
x_1^i x_j x_1^{-i-1}, \substack{0\leq i \leq n-2 \\ 2 \leq j \leq s-1}
\\
x_1^{n-1}x_j, 1\leq j \leq s-1
\end{array}
$
& $Y_n$ & $\Z/n\Z$ & $R_n/R_n'$
\\
\hline
\end{tabular}
\caption{Generators and homology \label{Tab:hom}}
\end{table}

%
\subsection{The Burau Representation}
\label{sec:BurauDiscrete}
Consider the action of a generator $\sigma_i$ of $B_s$ seen as an automorphism of the free group, given for $1\leq i,j \leq s-2$ as
\[
\sigma_i(x_j)=
\begin{cases}
x_j & \text{ if } j\neq i, i+1 \\
x_i & \text{ if } j=i+1\\
x_i x_{i+1} x_i^{-1} & \text{ if } j=i
\end{cases}
\]
Therefore the conjugation action on the generators $\beta_j=x_j x_1^{-1}$ of $R_0$, seen as a $\mathbb{Z}[\mathbf{Z}]$-module, is given for $j\geq 2$ by:
\[
\sigma_j(\beta_{j+1})=\sigma_j(x_{j+1}x_1^{-1})=x_jx_1^{-1}=\beta_j,
\]
\begin{align*}
\sigma_j(\beta_j) &= \sigma_j(x_jx_1^{-1}) 
 = x_j \cdot x_{j+1} \cdot x_{j}^{-1} \cdot x_1^{-1} =  x_j x_1^{-1} \cdot x_1 x_{j+1} x_1^{-2} x_1^2  x_{j}^{-1} \cdot x_1^{-1} \\
 &=  \beta_j x_1\beta_{j+1}x_1^{-1} x_1\beta_j^{-1} x_1^{-1} 
= \beta_j \beta_{j+1}^t \beta_j^{-t}=\beta_j^{1-t} \beta_{j+1}^t.
\end{align*}
The notation for $t$ above is in accordance with the group algebra notation $\Z[\mathbf{Z}]=\Z[t,t^{-1}]$.
Also in the special case where $j=1$ we compute:
\[
\sigma_1(\beta_{2})=\sigma_1(x_{2}x_1^{-1})=x_1 \cdot x_1 x_2^{-1} x_1^{-1}=\beta_2^{-t},
\]
and  if  $i > 2$
\begin{align*}
\sigma_1(\beta_i) &= \sigma_1(x_ix_1^{-1}) 
=  x_i \cdot x_{1}  x_{2}^{-1}  x_1^{-1} 
 =  x_i x_1^{-1} \cdot x_1 x_{1} x_2^{-1}  x_{1}^{-1} 
 =  \beta_i  \beta_{2}^{-t}.
\end{align*}
We now compute the action on the $\mathbb{Z}[\mathbf{Z}]$-module $R/R'$, so the $\beta_i,\beta_j$ are commuting and we arrive at the matrix of the action with respect to the basis $\{\beta_2,\ldots,\beta_{s-1}\}$:
\[
\sigma_j \mapsto
\begin{pmatrix}
\mathrm{Id} &     &   & \\
            & 1-t & 1 & \\
            & t   & 0 &   \\
            &     &   & \mathrm{Id}
\end{pmatrix}, \text{if} \ \
\ j\neq 1
\text{ 
and }
\sigma_1 \mapsto
\begin{pmatrix}
-t          &  -t   &    & -t \\
 0          &   1   &    & 0 \\
 \vdots         &  \ddots  &  \ddots   &   \\
 0          &  \cdots    &  0   & 1
\end{pmatrix}.
\]
\begin{lemma} \label{CommActions}
The action of $t$ on $R_0^{\mathrm{ab}}$ commutes with the action of the braid group. 
\end{lemma}
\begin{proof}
It is obvious that for $\sigma_{j}$ $j\geq 2$ and $a\in R_0$ we have
\[
\sigma_j(a^t)=\sigma_{j}(x_1 a x_1^{-1})=x_1 \sigma_j(a) x_1^{-1}
=(\sigma_j(a))^t.
\]
For $\sigma_1$ we observe that
\begin{align*}
\sigma_1(a^t)&= \sigma_1(x_1 a x_1^{-1})=x_1 x_2 x_1^{-1} \sigma_1(a) x_1 x_2^{-1} x_1^{-1}=
x_1 \beta_2 \sigma_1(a) \beta_2^{-1} x_1^{-1} \\
&=x_1 \sigma_1(a) x_1^{-1}=(\sigma_1(a))^t,
\end{align*}
since $\sigma_1(a)$ is expressed as product of $\beta_\nu$ and the elements $\beta_i$ commute modulo $R_0'$.
\end{proof}

\subsection{The profinite Burau representation}
\label{sec:Burau-prof}

Since the action of elements $\sigma\in 
 \mathrm{Gal}(\bar{\Q}/\Q)$ on  elements $x_i$ involves $N(\sigma) \in \Z_\ell$, we cannot 
define 
an action of the  absolute Galois group on $H_1(C_s,\Z_\ell)=H_1(C_s,\Z) \otimes_\Z \Z_\ell=\Z_\ell[\mathbf{Z}]^{s-2}$, 
in the same way we  defined the action of the braid group on $H_1(C_s,\Z)$.

Recall that we denote by  $\mathbf{Z}_\ell$  the group $\Z_\ell$ written multiplicatively, i.e.
$\mathbf{Z}_\ell \cong \langle t^\alpha, \alpha\in \Z_\ell \rangle$.
It turns out that instead of the ordinary group algebra $\Z_\ell[\mathbf{Z}]$ we need the completed group algebra $\Z_\ell[[\mathbf{Z}_\ell]]$.

In this way we see the profinite Burau representation as a linear representation:
\[
\rho_{\mathrm{Burau}}: \mathrm{Gal}(\bar{\Q}/\Q) \rightarrow \mathrm{GL}_{s-2}(\Z_\ell[[\mathbf{Z}_\ell]]).
\]
\begin{remark}
The $\Z_\ell$-algebra $\Z_\ell[[\mathbf{Z}_\ell]]$ is a ring
defined as the inverse limit 
\[
\Z_\ell[[\mathbf{Z}_\ell]]=
\lim_{\substack{\leftarrow \\ n}} 
\Z_\ell[\Z/\ell^n \Z]
\]
of the ordinary group algebra, see \cite[p.171]{RibesZalesskii}.
It contains the $\Z$-algebra $\Z[\mathbf{Z}]\cong \Z[t,t^{-1}]$ which appears in the discrete topological Burau representation as a dense subalgebra. 
\end{remark}
\begin{lemma}
\label{existLimit} 
Let $\alpha=\sum_{\nu=0}^\infty a_\nu \ell^\nu \in \Z_\ell$, $0\leq a_\nu < \ell$ for all $0\leq \nu$. 
Set
\[
A_n=\left(
1+t+t^2+\ldots+ t^{(\sum_{\nu=0}^n a_\nu \ell^\nu)-1}
\right). 
\]
Then the sequence above converges and we will denote its limit by $(t^\alpha-1)/(t-1)$, that is 
\[
\lim_{n\rightarrow \infty} 
\left(
1+t+t^2+\ldots+ t^{(\sum_{\nu=0}^n a_\nu \ell^\nu)-1}
\right)
=
\frac{t^{\alpha}-1}{t-1}.
\]
\end{lemma}
\begin{proof}
The algebra $\Z_\ell[\Z/\ell^n \Z]$ is identified by the set of all expressions $\sum_{\nu=0}^{\ell^n-1} b_\nu t_n^\nu$, where $t_n$ is a generator of the cyclic group $\Z/\ell^n \Z$ and $b_\nu \in \Z_\ell$. In the inverse limit defining the ring of $\ell$-adic numbers the generator $t_{n+1}$ of 
$\Z/\ell^{n+1}\Z$ is sent to the generator $t_n$ of $\Z/\ell^n\Z$. The corresponding map in the group algebras (by identifying $t_n=t_{n+1}=t$) is given by sending 
\[
 \Z_\ell[\Z/\ell^{n+1}\Z] \ni
 \sum_{\nu=0}^{\ell^{n+1}-1} b_\nu t^\nu
 \longmapsto
\sum_{\nu=0}^{\ell^{n}-1} b_\nu t^\nu
\in 
\Z_\ell[\Z/\ell^{n}\Z].
\] 
We compute now for $m<n$
\begin{align*}
A_n-A_m &=
\sum_{\nu=a_0+a_1\ell+\cdots+ a_m \ell^m}^
{a_0+a_1\ell+\cdots+ a_n \ell^n} t^\nu  
= t^{a_0+a_1\ell+\cdots+ a_m \ell^m} 
\sum_{\nu=0}^{a_{m+1} \ell^{m+1} + \cdots +
a_n \ell^n
} t^\nu
\end{align*}
Therefore, the sequence is Cauchy and converges in the complete group algebra $\Z_\ell[[\mathbf{Z}_\ell]]$.
\end{proof}

\begin{lemma} \label{writeInv}
We have for $\alpha\in \mathbb{N}$, $\beta_k=x_k x_1^{-1}$.
\begin{equation}
\label{n-rel}
x_k^\alpha x_1^{-\alpha}= \beta_k \cdot \beta_k^t \cdot \beta_k^{t^2} \cdots \beta_k^{t^{\alpha-1}}.
\end{equation}
For $\alpha\in \Z_\ell$ we have
\begin{equation}
\label{n-rel1}
x_k^\alpha x_1^{-\alpha}= 
\beta_k^{\frac{t^\alpha-1}{t-1}}.
\end{equation}
\end{lemma}
\begin{proof}
We will prove first the result for $\alpha=n\in \mathbb{\Z}$.  
Indeed, for $\alpha=1$ the result is trivial while by induction
\[
x_k^n x_1^{-n}  =  x_k \beta_k \cdots \beta_k^{t^{n-2}} x_1^{-1} 
 =  x_k x_1^{-1} x_1   \beta_k \cdots \beta_k^{t^{n-2}} x_1^{-1} 
=  \beta_k \cdot \beta_k^t \cdot \beta_k^{t^2} \cdots \beta_k^{t^{n-1}}
\]
Now for $\alpha=\sum_{\nu=0}^{\infty} a_\nu \ell^{\nu} \in \Z_\ell$ we consider the 
sequence $c_n=\sum_{\nu=0}^{n} a_\nu \ell^{\nu}\rightarrow \alpha$. 
We have 
\[
x_k^\alpha x_1^{-\alpha} =\lim_n x_k^{c_n} x_1^{-c_n}=
\lim_n \beta_k^{\frac{t^{c_n}-1}{t-1}} =
\beta_k^{\frac{t^{\alpha}-1}{t-1}}.
\]
\end{proof}
\begin{lemma}
\label{pass-over}
For every $i\neq 1$, and $N\in \Z_\ell$ we have 
\[
x_i^{-1}x_1^{-N} =x_1^{-N} x_i^{-1} \cdot \beta_i^{1-t^N}.
\]
More generally  for $a\in \Z_\ell^*$
\[
x_i^{-a} x_1^{-N}= x_1^{-N} x_i^{-a} \cdot \beta_i^{\frac{t^a-1}{t-1}(1-t^N)}
\]
\end{lemma}
\begin{proof}
We compute
\begin{align*}
x_i^{-1}x_1^{-N} &=x_1^{-N} x_i^{-1} \cdot 
x_i  x_1^{N} x_i^{-1} x_1^{-N} 
\\
& =
x_1^{-N} x_i^{-1} \cdot
x_i x_1^{-1} x_1^{N} (x_i x_1^{-1})^{-1} x_1^{-N} \\
&= x_1^{-N} x_i^{-1} \cdot\beta_i \beta_i^{-t^N} \\
& =x_1^{-N} x_i^{-1} \cdot\beta_i^{1-t^N}. 
\end{align*}
The second equality is proved the same way
\begin{align*}
x_i^{-a}x_1^{-N} &=x_1^{-N} x_i^{-a} \cdot 
x_i^{a}  x_1^{N} x_i^{-a} x_1^{-N} 
\\
& =
x_1^{-N} x_i^{-a} \cdot
x_i^a x_1^{-a} x_1^{N} (x_i^a x_1^{-a})^{-1} x_1^{-N} \\
&= x_1^{-N} x_i^{-a} \cdot\beta_i^{\frac{t^a-1}{t-1}(1-t^N)}. 
\end{align*}
\end{proof}
\begin{lemma}
For a given word $x_{s-1}^{-a_{s-1}}\cdots x_{1}^{-a_{1}}$ we have
\[
\left(
x_{s-1}^{-a_{s-1}}\cdots x_{1}^{-a_{1}}
\right) x_1^{-N}=
x_1^{-N}
\left(
x_{s-1}^{-a_{s-1}}
\beta_{s-1}^{
 \frac{
 t^{a_{s-1}}-1
 }
 {t-1}
 (1-t^N)
 }
 \cdots
  x_{2}^{-a_{2}}
  \beta_{2}^
  {
  \frac{
  t^{a_{2}}-1}
  {t-1}
  (1-t^N)
  }
  x_1^{-a_1}
  \right).
\]
\end{lemma}
\begin{proof}
We use lemma \ref{pass-over} inductively to have
\begin{align*}
x_{s-1}^{-a_{s-1}}\cdots x_{1}^{-a_{1}} x_1^{-N}
&=
x_{s-1}^{-a_{s-1}}\cdots x_{3}^{-a_{3}}  x_1^{-N} 
x_{2}^{-a_{2}}
 \beta_{2}^{
 \frac{
 t^{a_{2}}-1
 }
 {t-1}
 (1-t^N)
 }
 x_1^{-a_1}
 \\
  &=
  x_{s-1}^{-a_{s-1}}\cdots x_{4}^{-a_{4}} 
   x_1^{-N} 
 x_{3}^{-a_{3}}
 \beta_{3}^{
  \frac{
  t^{a_{3}}-1
  }
  {t-1}
  (1-t^N)
  }
x_{2}^{-a_{2}}
 \beta_{2}^{
 \frac{
 t^{a_{2}}-1
 }
 {t-1}
 (1-t^N)
 }
 x_1^{-a_1}
 \\
&= \cdots
\\
&= 
x_1^{-N}
x_{s-1}^{-a_{s-1}}
\beta_{s-1}^{
 \frac{
 t^{a_{s-1}}-1
 }
 {t-1}
 (1-t^N)
 }
 \cdots
 x_{2}^{-a_{2}}
 \beta_{2}^{
 \frac{
 t^{a_{2}}-1
 }
 {t-1}
 (1-t^N)
 }
 x_1^{-a_1}.
\end{align*}
\end{proof}
For simplicity denote $N(\sigma)$ by $N$ and $w_i(\sigma)$ by $w$. 
We will consider $w x_i^N w^{-1} x_1^{-N}$, where
$
w^{-1}=x_{s-1}^{-a_{s-1}}\cdots x_{1}^{-a_{1}}
$.
We have
\begin{align*}
wx_i^N w^{-1} x_1^{-N} &=
\beta_i^{t^{\sum_{\nu=1}^{s-1} a_{\nu}}\frac{t^N-1}{t-1}}
\beta_{s-1}^
{
	t^{\sum_{\nu=1}^{s-2} a_{\nu}}\frac{t^{a_{s-1}}-1}{t-1}
(1-t^N)
}
\cdots
\beta_{2}^ 
{
	t^{a_1} \frac{t^{a_2}-1}{t-1}
(1-t^N)
}.
\end{align*}
An arbitrary element $w \in \mathfrak{F}_{s-1}$  can be written in a unique way as
\[
w=B\cdot x_1^{a_1}\cdots x_{s-1}^{a_{s-1}},  \qquad a_i \in \Z_\ell
\]
where $B$ is an element in the group $R_0$ generated by the elements $\beta_i$, $i=2,\ldots,s-1$. Observe now that 
for every $\beta_i$, and $N\in \Z_{\ell}$ we have 
\[
\beta_i x_1^{-N} = x_1^{-N} x_1^{N} \beta_i x_1^{-N}=
x_1^{-N} \beta_i^{t^N}.  
\]
By considering a sequence of words in $\beta_i$ tending to $B$ we see that
\[
B x_1^{-N} = x_1^{-N} B^{t^N}, 
\]
for every element $B$ in the pro-$\ell$ completion of $R_0$. 

This means that 
\[
w x_i^{N} w^{-1} x_1^{-N} 
= B (x_1^{a_1}\cdots x_{s-1}^{a_{s-1}}) x_i^N 
 (x_{s-1}^{-a_{s-1}} \cdots x_1^{-a_1}) B^{-1} x_1^{-N}
\]
\[
{\scriptstyle
=B 
\left(
x_1^{a_1}\cdots x_{s-1}^{a_{s-1}}
\right) x_i^N x_1^{-N}
\left(
x_{s-1}^{-a_{s-1}} 
\beta_{s-1}^{(1-t^N)
\frac{t^{
a_{s-1}}-1
}
{t-1}
}
\cdots
x_{2}^{-a_2}
\beta_{2}^{(1-t^N)
\frac{
t^{a_{2}}-1
}
{t-1}
}
x_1^{-a_1}
\right)
B^{-t^N} 
}
\]
\[=
B
\beta_i^{t^{a_1+\cdots a_{s-1}}
\frac{t^{
N}-1
}
{t-1}
}
\beta_{s-1}^{(1-t^N)t^{a_1+\cdots +a_{s-2}}
\frac{t^{
a_{s-1}}-1
}
{t-1}
}
\cdots
\beta_2^{(1-t^N) t^{a_1}
\frac{t^{
a_2}-1
}
{t-1}
}
B^{-t^N}.
\]
The above in $R_0/R_0'$ evaluates to 
\begin{equation}
{\scriptstyle
\label{congAbsQ}
w x_i^N w^{-1} x_1^{-N}=
\beta_i^{t^{a_1+\cdots+ a_{s-1}}
\frac{t^{
N}-1
}
{t-1}
}
\beta_{s-1}^{(1-t^N)t^{a_1+\cdots +a_{s-2}}
\frac{t^{
a_{s-1}}-1
}
{t-1}
}
\!\!\!\!\!\!\cdots
\beta_2^{(1-t^N)t^{a_1}
\frac{t^{
a_2}-1
}
{t-1}
}
B^{-t^N+1}.
}
\end{equation}

\begin{theorem}
\label{MatBurau}
For  $\sigma \in \mathrm{Gal}(\bar{\Q}/\Q)$ and $1\leq i \leq s-1$ we have that $\sigma(x_i)=w_i(\sigma) x_i^{N(\sigma)} w_i(\sigma)^{-1}$, where $N(\sigma)$ is the cyclotomic character $N:\mathrm{Gal}(\bar{\Q}/\Q)\rightarrow \Z_\ell^*$. 
Consider the multiplicative group $\mathbf{Z}_\ell$ which is isomorphic to $\Z_\ell$ and has 
topological generator $t$ given by
$
\mathbf{Z}_\ell\cong 
\langle
t^\alpha, \alpha\in \Z_\ell
\rangle.
$ 
Let us write  
\[
w_i(\sigma)= B_i(\sigma)x_1^{a_{1,i}(\sigma)}\cdots x_{s-1}^{a_{s-1,i}(\sigma)}, \qquad a_{\nu,i}(\sigma)\in \Z_\ell,
\]
where $B_i(\sigma) \in R_0/R_0'$ is expressed as
\[
B_i(\sigma)=
\beta_2^{b_{2,i}(\sigma)}
\cdots
\beta_{s-1}^{b_{s-1,i}(\sigma)}C,
\] 
with $b_{i,j}(\sigma) \in \Z_\ell$ and $C\in R_0'$. 
The matrix representation of $\rho_{\mathrm{Burau}}$ with respect to the basis $\beta_j=x_jx_1^{-1}$, $j=2,\ldots,s-1$ has the following form:
\[
\rho_{\mathrm{Burau}}(\sigma)=
\frac{t^{N(\sigma)}- 1}{t-1}
 L(\sigma)
+
\big(
1-t^{N(\sigma)}
\big) M(\sigma)+
\big(
1-t^{N(\sigma)}
\big) K(\sigma),
\]
where 
$L,M,K$ are $(s-2)\times(s-2)$ matrices given by
\[
L(\sigma)=
\mathrm{diag}
\left(
t^{\sum_{\nu=1}^{s-1} a_{\nu,2}(\sigma)},
\ldots
,
t^{\sum_{\nu=1}^{s-1} a_{\nu,s-2}(\sigma)}
\right)
\]
\[
M(\sigma)\!\!= \!\!{\scriptscriptstyle
\begin{pmatrix}
\Gamma(a_{2,2}) \cdot
t^{a_{1,2}(\sigma)}  
& \cdots &
\Gamma(a_{s,s-1}) \cdot
t^{a_{1,s-1}(\sigma)} \\
\Gamma(a_{3,2}) \cdot
t^{a_{1,2}(\sigma)+a_{2,2}(\sigma)}  
& \cdots &
\Gamma(a_{3,s-1}) \cdot
t^{a_{1,s-1}(\sigma)+a_{2,3}(\sigma)} \\
\vdots  
&  & \vdots \\
\Gamma(a_{s-2,2}) \cdot
t^{a_{1,2}(\sigma)
+\cdots+a_{s-1,2}(\sigma)} 
 & \cdots &
 \Gamma(a_{s-2,s-1}) \cdot
t^{a_{1,s-1}(\sigma)
+\cdots+a_{s-1,s-1}(\sigma)}
\end{pmatrix}
}
\]
\[
K(\sigma)=
\begin{pmatrix}
b_{2,2}(\sigma) & b_{2,3}(\sigma)  & \cdots & b_{2,s-1}(\sigma) \\
b_{3,2}(\sigma) & b_{3,3}(\sigma)  & \cdots & b_{3,s-1}(\sigma) \\
\vdots & \vdots &  & \vdots \\
b_{s-1,2}(\sigma) & b_{s-1,3}(\sigma)  & \cdots & b_{s-1,s-1}(\sigma)
\end{pmatrix}.
\]
In the above theorem the term
 \[
\Gamma(a):=(t^a-1)/(t-1)\]
 for $a\in \Z_\ell$,
 is defined in lemma \ref{existLimit}.
\end{theorem}
\begin{proof}
We will find the matrix $\rho$ corresponding  to the action given by 
$\sigma(x_i)=w_i(\sigma) x_i^{N(\sigma)} w_i(\sigma)^{-1}$.
Let us write each $w_i(\sigma)$
as 
\[
w_i(\sigma)= B_i(\sigma)x_1^{a_{1,i}(\sigma)}\cdots x_{s-1}^{a_{s-1,i}(\sigma)}, 
\]
where $B_i(\sigma) \in R_0/R_0'$ is expressed as
\[
B_i(\sigma)=
\beta_2^{b_{2,i}(\sigma)}
\cdots
\beta_{s-1}^{b_{s-1,i}(\sigma)}C,
\] 
with $b_{i,j}(\sigma) \in \Z_\ell[[\mathbf{Z}_\ell]]$ and $C\in R_0'$. 

Let us now consider the action of $\sigma$ on $\beta_i$ for $i=2,\ldots,s-1$ and recall that just after eq. (\ref{actGeneratos}) we have selected a  normalization by an inner automorphism $w_1(\sigma)=1$, 
so that $\sigma(x_1)=x_1^{N(\sigma)}$.
Therefore
\[
\sigma(\beta_i)=\sigma(x_i x_1^{-1})=
 w_i(\sigma) x_i^{N(\sigma)} w_i(\sigma)^{-1} x_1^{-N(\sigma)}.
\]
The matrix form of $\rho_{\mathrm{Burau}}$ as given in theorem \ref{MatBurau} follows by eq. (\ref{congAbsQ}).
More preciselly the 
 matrix $L(\sigma)$ comes from the coefficients of the factor $\beta_i^{a_1+\cdots+a_{s-1}}$, the matrix $M(\sigma)$ comes from the next factor
 \[\beta_{s-1}^{(1-t^N)t^{a_1+\cdots+a_{s-2}} 
(t^{a_{s-1}}-1)/(t-1) 
 }\cdots
 \beta_2^{(1-t^N)t^{a_1} 
(t^{a_{2}}-1)/(t-1)
 } \]
  and the matrix $K(\sigma)$ comes from the final factor $B^{-t^N+1}$.
\end{proof}

\section{Examples - Complete curves}
\label{sec:applications-cyc-cov}

\subsection{The compactification of cyclic covers}
Every topological cover of the Riemann surface $\mathbb{P}^1\backslash \{P_1,\ldots,P_{s}\}$ gives rise to a Riemann surface $X^0$, which can compactified to a compact Riemann surface $X$, see \cite[prop. 19.9]{MR1343250}. Moreover if the topological cover is Galois with Galois group $G$, then the corresponding function field $\mathbb{C}(X)/\mathbb{C}(x)$ form a Galois extension with the same Galois group. We  know that every Kummer extension of the rational function field, totally ramified above $s$ points, 
corresponds to the cyclic cover of the projective line given by:
\begin{equation} \label{cyccov}
y^n=\prod_{i=1}^{s} (x-b_i)^{d_i}, \qquad (d_i,n)=1.
\end{equation}
For different choices of exponents $d_1,\ldots,d_s$ the curves are in general not isomorphic, see \cite{Kallel-Sjerve}. 
Without loss of generality we can assume that the infinity point of this model is not ramified and this is equivalent to the condition 
 $\sum\limits_{i=1}^s d_i\equiv 0 \mod n$,  see \cite[p. 667]{Ko:99}. This means the ramified points $\{P_1,\ldots,P_{s-1},P_s=\infty\}$ in our original setting are now mapped to the points $\{b_1,\ldots,b_s\}$. 

Conversely, the cover given in eq. (\ref{cyccov})
determines equivalently a cyclic Kummer extension of the rational function field $\mathbb{C}(x)$ and since the exponents $d_i$ are prime to $n$ we have that the points $P_1,\ldots,P_{s}$ are all fully ramified see \cite{Ko:99}. Therefore, the open curve obtained by removing the $s$ points $Q_1,\ldots,Q_s$ which map onto $P_1,\ldots,P_s$ is a topological cyclic cover, which can be considered with the tools developed so far. 

However, we  will show that the assumption made so far in this article lead to the selection $d_i=1$ for all $1\leq i \leq s-1$. 
Let $Q_i$ be the unique point of $X$ above $b_i$ and let $t_i$ be a local uniformizer at $Q_i$. 
We can select  $t_i$ so that 
$x-b_i=t_i^n$. Indeed, valuation of $x-b_i$ in the local ring at $Q_i$ is $n$ and by Hensel's lemma any unit is an $n$-power that can be absorbed by reselecting the uniformizer $t_i$ if necessary.
We can replace the factor $(x-b_i)^{d_i}$ in the original defining equation (\ref{cyccov}) of the curve in order to arrive at the following equation
\begin{equation}
\label{eq1d}
y^n= t_i^{nd_i} U, \qquad 
U=
\prod_{\substack{\nu=1 \\ \nu \neq i}}^{s} (x-b_\nu)^{d_\nu}
\in k[x], v_{Q_i}(U)=0.
\end{equation}
The element $U$ is invariant under the action of $\langle \sigma \rangle$ and so is its $n$-th root $u\in k[[t_i]]$. Indeed, since $\sigma(u^n)=\sigma(U)=u^n$ we have that $\sigma(u)=\zeta^a u$, for some $a$, $0\leq a <n$. But $u$ is a unit, therefore $u \equiv a_0 \mod t_i k[[t_i]]$, for some element $a_0 \in k$, $a_0\neq 0$. Also $\sigma(a_0)=a_0$, so by considering  $\sigma(u)=\zeta^a u$ modulo $t_i k[[t_i]]$ we obtain  $a_0= \zeta^a a_0 $. This implies that  $a=0$ and $u$ is a $\sigma$-invariant element.

Since $x-b_i=t_i^n$
the generator $\sigma$ of $\mathrm{Gal}(X/\mathbb{P}^1)$ acts on $t_i$ by sending $\sigma(t_i)=\zeta^{\ell} t_i$ for some $\ell \in \mathbb{N}$. 
This $\ell$ equals $d_i^*$ for some $0< d_i^* < n$, where 
  $d_i d_i^*\equiv 1 \mod n$. 
  Indeed, 
by taking the $n$-root in eq. (\ref{eq1d}), we have 
$
y=t_i^{d_i} u
$
for some $\sigma$-invariant unit in $k[[t_i]]$. Then the action of $\sigma$ gives us that 
$\zeta=\zeta^{\ell d_i}$, so $\ell d_i \equiv 1 \mod n$.
So in the short exact sequence 
\[
1 \rightarrow 
R_n 
\rightarrow 
F_{s-1}
\rightarrow
\Z/n\Z
\rightarrow 
1
\]
the elements $x_i$, which correspond to loops winding once around each branch point, map to the element $\sigma^{d_i^*} \in \Z/n\Z$. This is not compatible  
with the selection of the winding number function $\alpha$ given in equation (\ref{w-map}) unless all $d_i$ are equal. Without loss of generality we can assume that $d_i=1$ for all $1\leq i \leq s-1$.

Riemann-Hurwitz theorem implies that
\begin{equation} \label{genusCyclic}
g=\frac{(n-1)(s-2)}{2},
\end{equation}
which is compatible with the computation of $r=2g+s-1$ given in eq. (\ref{ASRn}).

This curve can be uniformized as a quotient $\mathbb{H}/\Gamma$ of the hyperbolic space modulo a discrete free subgroup of genus $g$, which admits a presentation
\[
\Gamma=\langle a_1,b_1,a_2,b_2,\ldots,a_g,b_g | [a_1,b_1][a_2,b_2]\cdots [a_g,b_g]=1 \rangle.
\]
On the other hand side, when we remove the $s$ branch points we obtain a topological cover of the space $X_s$ defined in the previous section.
This topological cover corresponds to the free subgroup  $R_n<F_{s-1}$ given by
\[
R_n=\langle  a_1,b_1,a_2,b_2,\ldots,a_g,b_g,\gamma_1,\ldots,\gamma_s | \gamma_1 \gamma_2 \cdots \gamma_s \cdot[a_1,b_1]\cdots [a_g,b_g]=1\rangle.
\]

The group $\mathrm{Gal}(X/\mathbb{P}^1)\cong \Z/n\Z=\langle \sigma \rangle$ is a subgroup of the automorphism group $\mathrm{Aut}(X)\subset \mathrm{Mod}(X)$. Therefore the generator  $\sigma$ acts on $R_n$. 

Since by lemma \ref{lemma:funRn} the group $R_n$ is the fundamental group of $Y_n$ the space $R_n/R_n'$ is the first homology group of the open curve $Y_n$. By proposition \ref{prop:GalModRn} its structure is given by $H_1(Y_n,\Z)=\Z[\Z/n\Z]^{s-1} \bigoplus \Z$. 

Let $\widehat{R}_n,\widehat{R_n'}$ be the pro-$\ell$ completions of $R_n$ and $R_n'$ respectively. Since the quotient $R_n/R_n'$ is torsion free, the completion functor is exact, see \cite[p. 35 exer. 21,22]{DDMS} and \cite[p. 81-85]{RibesZalesskii}.
This allows us to see that 
\[
\widehat{R}_n /\widehat{R_n'}=\widehat{H_1(Y_n,\Z)}=H_1(Y_n,\Z_\ell). 
\]

\begin{lemma}
\label{invariantEll}
With notation as above, the $\mathrm{Gal}(X/\mathbb{P}^1)$-invariant elements of $H_1(Y_n,\Z)$ (resp. $H_1(Y_n,\Z_\ell)$) is the group generated by the elements
\[
\{x_i^n : 1\leq i \leq s-1\}.
\]
\end{lemma}
\begin{proof}
We will use the decomposition of proposition \ref{prop:GalModRn} for $H_1(Y_n,\Z)$ and the corresponding decomposition of $H_1(Y_n,\Z_\ell)=H_1(Y_n,\Z)\otimes_\Z \Z_\ell$.
Observe that an element in the group algebra $\mathbb{Z}[\langle \sigma \rangle]$ is $\sigma$-invariant if and only if it is of the form
$
\sum_{i=0}^{n-1}a\sigma^i$ for some $a\in \Z$.
Hence the invariant elements are multiples (powers in the multiplicative notation) by
\[
\beta_j \beta_j^\sigma \beta_j^{\sigma^2} \cdots \beta_j^{\sigma^{n-1}}=x_j^nx_1^{-n}.
\]
The action of 
$\langle \sigma\rangle=\mathrm{Gal}(X/\mathbb{P}^1)$ is 
given by conjugation with $x_1$, therefore $x_1^n$ is invariant under this conjugation action and the result follows.
\end{proof}

The elements $\gamma_i$ are lifts of the loops $x_i$ around each hole in the projective line. Thus $\gamma_i$ are $\mathbb{Z}/n\mathbb{Z}$-invariant.
Set $\gamma_i=x_i^{n}$.
The quotient
 $\mathbb{Z}[\mathbb{Z}/n\mathbb{Z}]/ \langle \sum_{i=0}^{n-1} \sigma ^i\rangle$ is the co-augmentation module, see  \cite[sec. 1]{NeukirchBonn}.

\begin{lemma}
We have
\[
x_k^n x_i x_k^{-n} x_1^{-1}= \beta_k \cdot \beta_k^\sigma  \cdot \beta_k^{\sigma^2} \cdots \beta_k^{\sigma ^{n-1}}
\cdot
\beta_i^{\sigma ^n}
\cdot
\beta_k^{-\sigma ^n}
\cdot
\beta_k^{-\sigma ^{n-1}}
\cdots
\beta_k^{-\sigma ^2}
\cdot
\beta_k^{-\sigma }
\]
Moreover in the abelian group $R/R'$ we have
\[
x_k^n x_i x_k^{-n} x_1^{-1}=\beta_i^{\sigma ^n} \beta_k^{1-\sigma ^n}.
\]
\end{lemma}
\begin{proof}
Write
\begin{eqnarray*}
x_k^n x_i x_k^{-n} x_1^{-1} & = &
x_k^n  x_1^{-n} \cdot x_1^{n}
x_i
x_1^{-1}
x_1^{-n}
x_1^{n+1}
x_k^{-n}
x_1^{-1} \\
& = &
\beta_k \cdot \beta_k^\sigma  \cdot \beta_k^{\sigma ^2} \cdots \beta_k^{\sigma ^{n-1}}
\cdot
x_1^{n} \beta_i x_1^{-n}
x_1
\left(
\beta_k \cdot \beta_k^\sigma  \cdot \beta_k^{\sigma ^2} \cdots \beta_k^{\sigma ^{n-1}}
\right)^{-1} x_1^{-1}
\\
& =&
\beta_k \cdot \beta_k^\sigma  \cdot \beta_k^{\sigma ^2} \cdots \beta_k^{\sigma ^{n-1}}
\cdot
\beta_i^{\sigma ^n}
\cdot
\beta_k^{-\sigma ^n}
\cdot
\beta_k^{-\sigma ^{n-1}}
\cdots
\beta_k^{-\sigma ^2}
\cdot
\beta_k^{-\sigma }
\end{eqnarray*}
\end{proof}

\begin{lemma}
The subgroup of $H_1(Y_n,\Z)=R_n/R_n'$   generated by the following two sets of $\mathbb{Z}/n\mathbb{Z}$-invariant elements 
 \[\{x_1^n, x_j^n x_1^{-n}: 2 \leq j \leq s-1\},
\{
x_j^n : 1 \leq j \leq s-1
\}
 \]
 is invariant under the action of the braid group.

The subgroup of $H_1(Y_n,\Z_\ell)$ generated by the same elements is invariant under the braid group 
  and under the action of the group $\mathrm{Gal}(\bar{\Q}/\Q)$.
\end{lemma}
\begin{proof}
We consider first the braid action. The proof is the same in the discrete and in the pro-$\ell$ setting.  
By lemma \ref{writeInv} we have
\begin{align*}
\sigma_1 (x_1^n) & =  (x_1 x_2 x_1^{-1})^{n} 
=  x_1 \cdot x_2^{n} \cdot x_1^{-1} 
=  x_1 \cdot x_2^{n} x_1^{-n}\cdot x_1^{n-1} 
\\
&=  x_1 \cdot \beta_2 \cdot \beta_2^\sigma \cdot \beta_2^{\sigma^2} \cdots \beta_2^{\sigma^{n-1}} \cdot x_1^{-1} \cdot x_1^{n} 
 =  \beta_2^\sigma \cdot \beta_2^{\sigma^{2}} \cdots \beta_2^{\sigma^{n}} \cdot x_1^{n} 
 \\
&=  \beta_2 \cdot \beta_2^\sigma \cdots \beta_2^{\sigma^{n-1}} \cdot x_1^{n} 
 =  x_2^{n} x_1^{-n} \cdot x_1^{n}=x_2^n \\
\sigma_1 (x_2^n)  &=  x_1^{n},  \sigma_1 (x_i^n)  =  x_i^{n} \ (i>2).
\end{align*}
\begin{eqnarray*}
\text{For $j\geq 2$:  }
\sigma_j (x_j^n x_1^{-n}) & = & (x_j x_{j+1} x_j^{-1})^{n} x_1^{-n} 
 =  x_j \cdot x_{j+1}^{n} \cdot x_j^{-1} \cdot x_1^{-n}\\
& = & x_jx_1^{-1} \cdot x_1 (x_{j+1}^{n}x_1^{-n}) x_1^{-1} \cdot x_1^{n} \cdot x_1x_j^{-1}\cdot x_1^{-n} \\
&= & x_{j+1}^n x_1^{-n} \\
\sigma_j(x_j^n)&= & \sigma_j(x_j^{n} x_1^{-n}) \sigma_j(x_1^n)
=x_{j+1}^n.
\end{eqnarray*}

We will now consider the action of $\mathrm{Gal}(\bar{\Q}/\Q)$, which makes sense only in the pro-$\ell$ setting. Each element  $\tau\in 
\mathrm{Gal}(\bar{\Q}/\Q)$ acts on $x_i$ by 
\[
\tau(x_i)=w_i(\tau) x_i^{N(\tau)} w_i(\tau)^{-1}, 
\]

Therefore, for $i=2,\ldots,s-1$ we have
\begin{align*}
\tau(x_i^n x_1^{-n}) &= 
\tau(\beta_j \beta_j^\sigma \cdots \beta_j^{\sigma^{n-1}}) \\
&= 
\left(
 \tau(\beta_j)\right)^{
 1+\sigma+\cdots+\sigma^{n-1}
 }\end{align*}
which is an element invariant under the action of $\Z/n\Z=\langle \sigma \rangle$,  therefore it belongs to the desired group by lemma \ref{invariantEll}. We have assumed that we will normalize by an inner automorphism the element $\tau$ so that $\tau(x_1^n)=x_1^{N(\tau)n}$, that is $w_1(\tau)=1$.
\end{proof}

Consider now the space
\[
H_1(\bar{Y}_n,\mathbb{Z})=\frac{R_n}{R_n' \cdot \langle \gamma_1,\ldots,\gamma_s  \rangle} =
\frac{R_n}{R_n' \cdot \langle x_1^n,\ldots,x_s^n \rangle}.
\]
Observe that $R_n/R_n' \cdot \langle x_1 \rangle=\mathbb{Z}[\mathbb{Z}/n\mathbb{Z}]^{s-2}.$
Since $\langle \gamma_1,\ldots,\gamma_s\rangle$ is both $\mathbb{Z}/n\mathbb{Z}$ and $B_s$ stable we have a natural defined action of $B_s$ on the quotient.
We compute now the action of the braid group on $\beta_j^{\sigma^i}=x_1^{i} x_j x_1^{-i-1}$. We can pick as a basis of the $\mathbb{Z}$-module $H_1(\bar{Y}_n,\mathbb{Z})$ the elements
\[
\{
\beta_j^{\sigma^i}=x_1^i x_j x_1^{-1-i}: 2\leq j \leq s-1, 0\leq i \leq n-2
\}
\]
and equation (\ref{n-rel}) written additively implies that 
$\beta_j^{\sigma^{n-1}}=-\sum_{\nu=0}^{n-2} \beta_j^{\sigma^\nu}$, recall that all powers $x_i^n$ are considered to be zero.


Let $J_{\Z/n\Z}$ be the co-augmentation module. 
 Observe that
$\beta_j^{t^{\nu}-1}=[x_1^{\nu},x_j]$.
It is well known (see, \cite[Prop. 1.2]{NeukirchBonn}) that 
$
\mathbb{Z}[\Z/n\Z]=J_{\Z/n\Z} \oplus \Z. 
$
We have
\begin{equation}
\label{DSeq}
H_1(\bar{Y}_n,\mathbb{Z})=J_{\Z/n\Z}^{s-2}.
\end{equation}
Notice that the above $\Z$-module has the correct rank $2g=(n-1)(s-2)$. 
The direct sum in eq. (\ref{DSeq}) is in the category of $\Z$-modules not in the category of $B_s$-modules. Also on the co-augmentation module $J_{\Z/n\Z}$ the generator of the $\Z/n\Z$ is represented by the matrix:
\begin{equation} \label{augmentationMat}
A:=
\begin{pmatrix}
0 & \cdots & 0 & -1 \\
1   & \ddots & \vdots & \vdots   \\
0 & \ddots & 0  & -1 \\
0 & 0 & 1 & -1
\end{pmatrix}
\end{equation}
which is the companion matrix of the polynomial $x^{n-1}+\cdots+ x+1$. 
Notice that for $n=p$ prime we can represent
 $J_{\Z/n\Z}$ is in terms of the $\Z$-module $\Z[\zeta]$, where $\zeta$ is a primitive $p$-th root of unity, i.e.
\[
\Z[\zeta]=\bigoplus_{\nu=0}^{p-2} \zeta^{\nu} \Z,
\]
and the $\Z[\Z/n\Z]$-module structure is given by multiplication by $\zeta$. 

Since the $\Z/n\Z$-action and the braid action are commuting we have a decomposition (notice that $1$ does not appear in the eigenspace decomposition below)
\[
H_1(\bar{Y}_n,\mathbb{Z})\otimes_{\Z}\C =\bigoplus_{\nu=1}^{n-1} V_\nu
\]
where $V_\nu$ is the eigenspace of the $\zeta^\nu$-eigenvalue. Each $V_\nu$
is a $B_s$-module of dimension $s-2$.
In order to compute the spaces $V_\nu$ we have to diagonalize the matrix given in eq. (\ref{augmentationMat}). Consider the Vandermonde matrix given by:
\[
P=
\begin{pmatrix}
1 & \zeta_1 & \zeta_1^2 & \cdots & \zeta_1^{n-2} \\
1 & \zeta_2 & \zeta_2^2 & \cdots & \zeta_2^{n-2}  \\
\vdots & \vdots &  & \vdots \\
1 & \zeta_{n-1} & \zeta_{n-1}^2 & \cdots & \zeta_{n-1}^{n-2}
\end{pmatrix},
\]
where $\{\zeta_1,\ldots,\zeta_{n-1}\}$ are all $n$-th roots of unity different than $1$. Observe that 
\[
P\cdot A = \mathrm{diag}(\zeta_1,\zeta_2,\ldots,\zeta_{n-1}) \cdot P.
\]
Thus the action of the braid group on the eigenspace $V_\nu$ of the eigenvalue $\zeta^\nu$ can be computed by a base change as follows:
Consider the initial base 
$1,\beta_j, \beta_j^{t},\ldots,\beta_j^{t^{n-2}}$ for $2 \leq j \leq s-1$. The eigenspace of the $\zeta^\nu$ eigenvalue has as basis  the $k$-elements  of the $1\times (n-2)$ matrix 
\[
\left( 1,\beta_j, \beta_j^{\sigma},\ldots,\beta_j^{\sigma^{n-2}} \right) \cdot P^{-1}
\]
for all $j$ such that 
$2\leq j \leq s-1$. These elements are $\C$-linear combinations of the elements $\beta_j$ and the action of the braid generators on them can be easily computed. 

Since the action of $\mathrm{Gal}(\bar{Y}_n/\mathbb{P}^1)=\langle \sigma \rangle$ commutes with the action of $B_{s}$ (resp. $\mathrm{Gal}(\bar{\Q}/\Q)$) each eigenspace is a $B_{s}$ (resp. $\mathrm{Gal}(\bar{\Q}/\Q)$) module. The action of the operator $t$ on each $V_n$ is essentially the action of $\sigma$, which by definition of eigenspace,  acts by multiplication by $\zeta_\nu$.
Therefore, the matrix representation corresponding to each eigenspace $V_n$ is the matrix of the Burau (resp. pro-$\ell$ Burau) evaluated at $t=\zeta_\nu$.

Similarly in the pro-$\ell$ case we have
\begin{equation}
\label{module-decomp1}
\Z_{\ell}[[\mathbf{Z}_\ell]]^{s-2}
\otimes_{\Z_\ell} \bar{\Q}_\ell = \bigoplus_{\nu=1}^{\ell^k-1} V_\nu,
\end{equation}
which after reducing $\Z_{\ell}[[\mathbf{Z}_\ell]] \rightarrow \Z_{\ell} [\Z_{\ell}/\ell^k \Z_\ell]=\Z_{\ell} [\Z/\ell^k\Z]$ sending $t\mapsto \zeta_\nu$ gives rise to the representation in $V_\nu$.

The decomposition in \ref{module-decomp1} is a decomposition of $\Z_\ell$-module. The Galois module structure and the $\Z_\ell$ action do not commute in this case. Indeed, the equation  (\ref{actGeneratos}) implies that 
$\sigma \in \mathrm{Gal}(\bar{\Q}/\Q)$ acts on the pro-$\ell$ generator by 
\[
\sigma t =t^{N(\sigma)} \sigma.
\]
Therefore, the modules $V_\nu$ defined above are $\mathrm{ker}N$-modules.  



\subsection{Relation to actions on holomorphic differentials}
\label{sec:GalInvarInterforms}

Let $S$ be a compact Riemann-surface of genus $g$. Consider the first homology group $H_1(S,\mathbb{Z})$ which is a free $\mathbb{Z}$-module of rank $2g$. Let $H^0(S,\Omega_S)$ be the space of holomorphic differentials which is a $\mathbb{C}$-vector space of dimension $g$. The function
\begin{eqnarray*}
H_1(S,\mathbb{Z}) \times H^0(S,\Omega_S) &  \rightarrow & \mathbb{R} \\
( \gamma , \omega)  & \mapsto & \langle \gamma, \omega \rangle =\mathrm{Re}\int_\gamma \omega
\end{eqnarray*}
induces a duality $H_1(S,\mathbb{Z})\otimes \mathbb{R}$ to $H^0(S,\Omega_S)^*$, see  \cite[th. 5.6]{LangIntroAlgAbFun}, \cite[sec. 2.2 p. 224]{Griffiths-Harris:95}. Therefore an action of a group element  on $H_1(S,\mathbb{Z})$ gives rise to the contragredient action on holomorphic differentials, see also \cite[p. 271]{Farkas-Kra}.

C. Mc Mullen in \cite[sec. 3]{McMullenBraidHodge} considered the Hodge decomposition of the DeRham cohomology as 
\[
H^1(X)=\mathrm{Hom}_{\C}(H_1(X,\Z),\C)=
H^{1,0}(X) \oplus H^{0,1}(X)\cong \Omega(X) \oplus \bar{\Omega}(X). 
\] 
Of course this decomposition takes place in the dual space of holomorphic differentials, and is based on the intersection form 
\begin{equation} \label{dua-intform}
\langle \alpha,\beta \rangle =i/2 \int_X \alpha \wedge \bar{\beta}, \qquad i^2=-1.
\end{equation}
In this article we use the group theory approach and we focus around the homology group $H_1(X,\Z)$. Homology group is equipped with an intersection form and a canonical symplectic basis $a_1,\ldots,a_g,b_1,\ldots,b_g$ such that 
\[
\langle a_i,b_j \rangle =\delta_{ij}, \qquad \langle a_i,a_j \rangle= \langle b_i,b_j \rangle =0.  
\] 
Every two homology classes $\gamma,\gamma'$ can be written as $\Z$-linear combinations of the canonical basis  
\[
\gamma = \sum_{i=1}^g (\lambda_i a_i + \mu_i b_i) \qquad
\gamma' = \sum_{i=1}^g (\lambda_i' a_i + \mu_i' b_i) 
\]
and the intersection is given by 
\[
\langle \gamma,\gamma' \rangle=
(\lambda_1,\ldots,\lambda_g, \mu_1,\ldots,\mu_g) 
\begin{pmatrix}
0 & \mathbb{I}_g  \\
- \mathbb{I}_g & 0 
\end{pmatrix}
(\lambda_1',\ldots,\lambda_g', \mu_1',\ldots,\mu_g')^t.
\]
This gives rise to a  representation
\begin{equation}
\label{sympRep}
\rho: B_{s-1} \rightarrow \mathrm{Sp}(2g,\mathbb{Z})
\end{equation}
since $\langle \sigma(\gamma),\sigma(\gamma')\rangle=\langle \gamma,\gamma'\rangle$. 
Indeed, it is known\cite[sec. 3.2.1]{MR2435235} that the action of the braid group   keeps the intersection multiplicity of two curves. 
 The relation to the unitary representation on holomorphic differentials (and the signature computations) is given by using the diagonalization of 
\[
\begin{pmatrix}
0 & \mathbb{I}_g  \\
- \mathbb{I}_g & 0 
\end{pmatrix}
=P \cdot  \mathrm{diag}(\underbrace{i,\ldots,i}_g,\underbrace{-i,\ldots,-i}_g) \cdot P^{-1}, 
\]
and the extra``$i$''  put in front of eq. (\ref{dua-intform}).

\subsubsection{Arithmetic intersection}
\label{arithInter}
In order to define an analogous result in the case of absolute Galois group we have first to define an intersection form in $H_1(X,\Z_\ell)$, which can be defined as the limit of the intersection forms in $H_1(X,\Z/\ell^n \Z)$.
 For every $\sigma\in \mathrm{Gal}(\bar{\Q}/\Q)$ and $\gamma,\gamma' \in H_1(X,\Z_\ell)$ we 
 have 
 \[
\langle \sigma(\gamma),\sigma(\gamma')\rangle=
\chi_\ell(\sigma) \langle \gamma,\gamma' \rangle,
 \]
 where $\chi_\ell(\sigma)$ is the $\ell$-cyclotomic character. 

Indeed, consider the Jacobian variety $J(X)$ for the curve $X$. By construction of the Jacobian variety as a quotient of its tangent space at the identity element it is clear that $H_1(J(X),\Z)=H_1(X,\Z)$ and after tensoring with $\Z_\ell$ the same equality  holds for the pro-$\ell$ homology groups. Consider the  following diagram
\[
\xymatrix{
	H_1(X,\Z) \times H_1(X,\Z) 
	\ar[r]^{\langle \cdot,\cdot\rangle} \ar[d] &
	\Z \ar[d]
	 \\
	T_\ell(J(X)) \times T_\ell(J(X)) 
	\ar[r]^{e^{\lambda}} &
	\Z_\ell(1)= 
	\displaystyle \lim_{\leftarrow} \mu_{\ell^n}\subset \bar{\Q},
}
\]
where the down horizontal array is given by the Weil pairing $e^{\lambda}$ with respect to the canonical polarization $\lambda$, and the upper map is the homology intersection form.
The arrows pointing down on the left are the obvious ones, while the down pointing arrow $\Z \rightarrow \displaystyle \lim_{\leftarrow} \mu_{\ell^n}$ 
is given by 
$\Z \ni m \mapsto (\ldots,e^{ \frac{2 \pi i m}{\ell^n} },\ldots  )$. 
The above diagram 
is known to commute with a negative sign,  see \cite[p. 237]{MumfordAbelian}, \cite[ex.13.3 p.58]{milneAV} that is 
\[
e^\lambda(a,a')= 
(\ldots, e^{ -\frac{2 \pi i \langle a,a' \rangle}{\ell^n} } ,\ldots )
\]
By selecting a primitive $\ell^n$-root of unity for every $n$, say $e^{2\pi i/\ell^n}$ we can write $\Z_\ell(1)$ as an additive module, that is we can send
\[
\Z_\ell(1) \ni \alpha =(\ldots, e^{2 \pi i a_n/ \ell^n},\ldots) \mapsto (\ldots, a_n,\ldots) \in \Z_\ell.
\]
It is known that the Weil pairing induces a symplectic pairing in $T_\ell(J(X))\cong H_1(X,\Z_\ell)$, \cite[prop. 16.6]{MR861974},\cite{MR3549183}, \cite{Duarte} so that 
\[
\langle \sigma a ,\sigma a' \rangle=
\chi_\ell(\sigma) \langle a, a'\rangle. 
\]
In this way we obtain a representation 
\[
\rho: \mathrm{Gal}(\bar{\Q}/\Q) \rightarrow 
\mathrm{GSp}(2g,\Z_\ell)
\]
which is the arithmetic analogue of the representation given in eq. (\ref{sympRep}).

 \def\cprime{$'$}

\end{document}